\documentclass{amsart}
\usepackage{amsfonts, amsbsy, amsmath, amssymb}

\ifx\fiverm\undefined 
  \newfont\fiverm{cmr5} 
\fi 

\input prepictex.tex
\input pictex.tex
\input postpictex.tex

\newtheorem{thm}{Theorem}[section]
\newtheorem{lem}[thm]{Lemma}
\newtheorem{cor}[thm]{Corollary}
\newtheorem{prop}[thm]{Proposition}

\newtheorem{conj}[thm]{Conjecture}

\newtheorem{rmk}[thm]{Remark}

\newtheorem{thm-con}[thm]{Theorem-Conjecture}
\numberwithin{equation}{section}

\theoremstyle{definition}

\allowdisplaybreaks

\newcommand{\f}{\Bbb F}
\begin{document}

\title{On a conjecture on permutation polynomials over finite fields}
\author{Wun-Seng Chou*}
\address{Institute of Mathematics, Academia Sinica, Taipei, Taiwan, ROC}
\email{macws@math.sinica.edu.tw}

\author[Xiang-dong Hou]{Xiang-dong Hou}
\address{Department of Mathematics and Statistics,
University of South Florida, Tampa, FL 33620}
\email{xhou@usf.edu}
\thanks{* The author is partially supported by the Ministry of Science and Technology, Taiwan, under the grant number 106-2115-M-001-003.}



\begin{abstract}
Let $\f_q$ be the finite field with $q$ elements and let $p=\text{char}\,\f_q$. It was conjectured that for integers $e\ge 2$ and $1\le a\le pe-2$, the polynomial $X^{q-2}+X^{q^2-2}+\cdots+X^{q^a-2}$ is a permutation polynomial of $\f_{q^e}$ if and only if (i) $a=2$ and $q=2$, or (ii) $a=1$ and $\text{gcd}(q-2,q^e-1)=1$. In the present paper we confirm this conjecture.
\end{abstract}

\maketitle

%%%%%%%%%%%%%%%%%%%%%%%%%%%%%%%%%%%%%%%%%%%
%  section 1
%%%%%%%%%%%%%%%%%%%%%%%%%%%%%%%%%%%%%%%%%%%
\section{Introduction}

Let $\f_q$ be the finite field with $q$ elements and let $p=\text{char}\,\f_q$. For integer $n\ge 0$, there is a polynomial $g_{n,q}\in\f_q[X]$ satisfying 
\begin{equation}\label{1.1}
g_{n,q}(X^q-X)=\sum_{a\in\f_q}(X+a)^n.
\end{equation}
The polynomial $g_{n,q}$ was introduced in \cite{Hou-JCTA-2011} as a generalization of the even characteristic {\em reversed Dickson polynomial} \cite{Hou-Mullen-Sellers-Yucas-FFA-2009}. A polynomial $f\in\f_q[X]$ is called a {\em permutation polynomial} (PP) of $\f_q$ if it induces a permutation of $\f_q$. The polynomial $g_{n,q}$ was studied in \cite{Fernando-Hou-Lappano-FFA-2013, Hou-FFA-2012} and the class turns out to be a rich source of PPs. The ultimate goal is to determine all PPs of $\f_{q^e}$ in the class $g_{n,q}$; this is a difficult question and a complete solution does not seem to be within immediate reach. The purpose of the present paper is to prove a conjecture about $g_{n,q}$.

It is known that for $a\ge 1$, $g_{q^{a+1}-2,q}=f_{a,q}$, where
\begin{equation}\label{1.2}  
f_{a,q}=X^{q-2}+X^{q^2-2}+\cdots+X^{q^a-2}.
\end{equation}

\begin{conj}[{\cite[Conjecture 5.1]{Fernando-Hou-Lappano-FFA-2013}}]\label{C1.1}
Let $e\ge 2$ and $1\le a\le pe-2$. Then $f_{a,q}$ is a PP of $\f_{q^e}$ if and only if 
\begin{itemize}
\item[(i)] $a=2$ and $q=2$, or
\item[(ii)] $a=1$ and $\text{\rm gcd}(q-2,q^e-1)=1$.
\end{itemize}
\end{conj} 

\begin{rmk}\label{R1.2}\rm
\begin{itemize}
\item[(i)]
When $e=1$, all PPs of $\f_{q^e}$ in the class $g_{n,q}$ have been determined in \cite{Fernando-Hou-Lappano-FFA-2013}. 
\item[(ii)]
Define $f_{0,q}=0$ and $f_{-1,q}=-X^{q^e-2}$. Then $f_{a,q}\equiv f_{r,q}\pmod{X^{q^e}-X}$, where $-1\le r\le pe-2$ is the remainder of $a$ modulo $pe$; hence we only need to consider $1\le a\le pe-2$.
\end{itemize}
\end{rmk}

\begin{rmk}\label{R1.3}\rm
Note that $f_{a,q} = X^{-2}L(X)^q$, where $L(X) = X + X^q +\cdots+ X^{q^{a-1}}$ is a $q$-linearized polynomial over $\f_q$. If $f_{a,q}$ is a PP of $\f_{q^e}$, then $0$ is the only root of $L(X)$ in $\f_{q^e}$, which happens if and only if $\text{gcd}(1 + X +\cdots+ X^{a-1}, X^e - 1) = 1$. The gcd condition is satisfied if and only if $a\not\equiv 0\pmod p$ and $\text{gcd}(a, e) = 1$, i.e., $\text{gcd}(a, pe) = 1$. Hence a necessary condition for $f_{a,q}$ to be a PP of $\f_{q^e}$ is that $\text{gcd}(a, pe) = 1$. 
\end{rmk}

When $q$ is even, the conjecture follows from Payne's theorem \cite{Hou-PAMS-2004, Payne-LRSFMN-1971}. In fact, when $q$ is even and $q\ge 4$,
\[
f_{a,q}=\Bigl(\frac{X^{q/2}+X^{q^2/2}+\cdots+X^{q^a/2}}X\Bigr)^2,
\]
where $X^{-1}(X^{q/2}+X^{q^2/2}+\cdots+X^{q^a/2})$ is a PP of $\f_{q^e}$ if and only if $a=1$ and $\text{gcd}(q/2-1,q^e-1)=1$. When $q=2$,
\[
f_{a,q}=\Bigl(1+\frac{X^{2}+X^{2^2}+\cdots+X^{2^{a-1}}}X\Bigr)^2,
\]  
where $X^{-1}(X^{2}+X^{2^2}+\cdots+X^{2^{a-1}})$ is a PP of $\f_{2^e}$ if and only if $a=2$. When $a=1$ and regardless of $q$, the conjecture is obviously true. Therefore, to confirm Conjecture~\ref{C1.1}, it remains to prove that when $q$ is odd, $e\ge 2$ and $2\le a\le pe-2$, $f_{a,q}$ is not a PP of $\f_{q^e}$. This claim is proved by combining two different approaches:

\begin{itemize}
\item[(1)] When $2\le a\le e/4$, the Hasse-Weil bound applied to $(f_{a,q}(X)-f_{a,q}(Y))/(X-Y)$ implies that $f_{a,q}$ is not a PP of $\f_{q^e}$. The proof of this claim can be found in a recent paper \cite{Hou-pre}; a brief recap of that proof is included in the present paper for the sake of completeness.

\item[(2)] For $e/4<a\le pe-2$, we use Hermite's criterion. The key is to choose a suitable integer $N$ ($0<N<q^e-1$) such that the coefficient of $X^{q^e-1}$ in the reduction of $f_{a,q}^N$ modulo $X^{q^e}-X$ is nonzero. The arguments in this part are quite involved. The choices of $N$ are the results of numerous experimentations; the most difficult cases are those with small $p$. To make the proof more readable, we introduced certain auxiliary multivariate polynomials to keep track of the involved computations that would otherwise require cumbersome notations. 
\end{itemize}

The paper is organized as follows: Section 2 contains some preliminaries which include notations, conventions and a technical lemma. The proof of Conjecture~\ref{C1.1} is given in Sections 3 -- 8, each dealing with a particular range of $a$:

Section 3: $2\le a\le e/4$;

Section 4: $e/4< a< e/3$;

Section 5: $e/3< a< e/2$;

Section 6: $e/2< a< e$;

Section 7: $e<a<(p-1)e$;

Section 8: $(p-1)e< a\le pe-2$.

\noindent By Remark~\ref{R1.3}, these ranges cover all possible values of $a$.

The two approaches mentioned above, the Hasse-Weil bound \cite{Hou-LFF-2018, Schmidt-1976, Stichtenoth-1993, Weil-1948} and Hermite's criterion \cite{Lidl-Niederreiter-FF-1997}, are well-known methods for studying PPs. The Hasse-Weil bound is particularly effective when dealing with PPs of finite fields with large $q$ \cite{Bartoli-FFA-2018, Cohen-AA-1970, Hou-pre, Leducq-DCC-2015, Williams-CMB-1968}. As for Hermite's criterion, the situation is somewhat different. There appears to be a tacit understanding that the criterion is easy to state but difficult to use; nontrivial applications of Hermite's criterion are rare. However, the work of the present paper suggests that Hermite's criterion may not have been utilized to its full strength previously; with due effort, it can offer solutions that other methods cannot.

%%%%%%%%%%%%%%%%%%%%%%%%%%%%%%%%%%%%%%%%%%%
%  section 2
%%%%%%%%%%%%%%%%%%%%%%%%%%%%%%%%%%%%%%%%%%%
\section{Preliminaries}

We maintain the following notations and conventions throughout the paper. 
For an integer $s\ge 0$, let $|s|_q$ denote its base-$q$ weight, i.e., for $s=\sum_{j\ge 0}s_jq^j$, where $s_j\in\{0,\dots,q-1\}$, $|s|_q=\sum_{j\ge 0}s_j$. A congruence equation $x\equiv y\pmod p$ is also written as $x\equiv_py$. 
For integer $N\ge 0$, let $C(N)$ be the coefficient of $X^{q^e-1}$ in the reduction of $f_{a,q}^N$ modulo $X^{q^e}-X$. By Hermite's criterion \cite[Theorem~7.4]{Lidl-Niederreiter-FF-1997}, if $C(N)\ne 0$ for some $0<N<q^e-1$, then $f_{a,q}$ is not a PP of $\f_{q^e}$. For $E\in\Bbb Z$, let $E^*$ be the integer such that $1\le E^*\le q^e-1$ and $E\equiv E^*\pmod{q^e-1}$.
For $F\in\f_q[X_0,\dots,X_n]$, the coefficient of $X_0^{\delta_0}\cdots X_n^{\delta_n}$ in $F$ is denoted by $[F:X_0^{\delta_0}\cdots X_n^{\delta_n}]$. 
The notation
\[
\sum_{\substack{\alpha,\beta,\cdots\cr (*),(\dagger),\cdots}}
\]
denotes a sum over  integers $\alpha,\beta,\cdots$ in a range specified by the context subject to conditions $(*),(\dagger),\cdots$. The multinomial coefficient $\binom n{n_1,\dots,n_k}$, where $n_i\ge 0$ and $n_1+\cdots+n_k=n$, is defined as 
\[
\binom n{n_1,\dots,n_k}=\frac{n!}{n_1!\cdots n_k!}.
\]
If $n_1+\cdots+n_k\ne n$, we define $\binom n{n_1,\dots,n_k}=0$.
Recall that if $n=\sum_{j\ge 0}n^{(j)}p^j$ and $n_i=\sum_{j\ge 0}n_i^{(j)}p^j$ are the base-$p$ expansions of $n$ and $n_i$, respectively, then \cite{Dickson-QJPAM-1902}
\[
\binom n{n_1,\dots,n_k}\equiv_p\prod_{j\ge 0}\binom {n^{(j)}}{n_1^{(j)},\dots,n_k^{(j)}}.
\]

\begin{lem}\label{L2.1}
Let $s_j$ and $\delta_j$ ($j\ge 0$) be nonnegative integers such that $s_j=\delta_j=0$ for all but finitely many $j$. Then
\begin{equation}\label{2.1}
\sum_{a_{ij}}\prod_{j\ge 0}\binom{s_j}{a_{0j},\dots,a_{mj}}=\Bigl[\prod_{j\ge 0}(X_j+X_{j+1}+\cdots+X_{j+m})^{s_j}:X_0^{\delta_0}X_1^{\delta_1}\cdots\Bigr],
\end{equation}
where the sum runs over all integers $a_{ij}\ge 0$ ($0\le i\le m$, $j\ge 0$) subject to the conditions 
\begin{align}\label{cond1}
\sum_{i=0}^m a_{ij}\,&=s_j,\quad j\ge 0,\\ \label{cond2}
\sum_{i+j=k}a_{ij}\,&=\delta_k,\quad k\ge 0.
\end{align}
\end{lem}

%%%%%%%%%%%%%%%%%%%%%%%%%%%%%%%%%%%%%%%%%%%%%%
%   begin pic
%%%%%%%%%%%%%%%%%%%%%%%%%%%%%%%%%%%%%%%%%%%%%%
\[
\beginpicture
\setcoordinatesystem units <6mm,6mm> point at 0 0

\arrow <5pt> [.2,.67] from 8 -4.3 to 8 -5 
\arrow <5pt> [.2,.67] from 3.7 -4.3 to 3 -5 

\setlinear
\plot 8 -0.3  8 -3.7 /
\plot 7.7 -0.3  4.3 -3.7 /
\plot 10 0.5  -0.8 0.5  -0.8 -4.5  10 -4.5  /

\put {$a_{00}$} at 0 0
\put {$a_{m0}$} at 0 -4
\put {$a_{m,j-m}$} at 4 -4
\put {$a_{mj}$} at 8 -4
\put {$a_{0j}$} at 8 0
\put {$s_j$} at 8 -5.5
\put {$\delta_j$} at 2.7 -5.5
\endpicture
\]
%%%%%%%%%%%%%%%%%%%%%%%%%%%%%%%%%%%%%%%%%%%%%%
%   end pic
%%%%%%%%%%%%%%%%%%%%%%%%%%%%%%%%%%%%%%%%%%%%%%

\begin{proof}
We have
\begin{align*}
\prod_{j\ge 0}(X_j+X_{j+1}+\cdots+X_{j+m})^{s_j}
=\,&\prod_{j\ge 0}\Bigl(\sum_{a_{0j}+\cdots+a_{mj}=s_j}\binom{s_j}{a_{0j},\dots,a_{mj}}\prod_{i=0}^m X_{i+j}^{a_{ij}}\Bigr)\cr
=\,&\sum_{\substack{a_{ij}\cr\eqref{cond1}}}\Bigl[\prod_{j\ge 0}\binom{s_j}{a_{0j},\dots,a_{mj}}\Bigr]\Bigl[\prod_{k\ge 0} X_k^{\sum_{i+j=k}a_{ij}}\Bigr].
\end{align*}
Hence
\[
\Bigl[\prod_{j\ge 0}(X_j+X_{j+1}+\cdots+X_{j+m})^{s_j}:\prod_{k\ge 0}X_k^{\delta_k}\Bigr]=\sum_{\substack{a_{ij}\cr\eqref{cond1},\eqref{cond2}}}\;\prod_{j\ge 0}\binom{s_j}{a_{0j},\dots,a_{mj}}.
\]
\end{proof}

From now on, $q$ is a power of an odd prime $p$, $e\ge 2$, $2\le a\le pe-2$ and $\text{gcd}(a,pe)=1$. The goal is to show that $f_{a,q}$ is not a PP of $\f_{q^e}$.

%%%%%%%%%%%%%%%%%%%%%%%%%%%%%%%%%%%%%%%%%%%
%  section 3
%%%%%%%%%%%%%%%%%%%%%%%%%%%%%%%%%%%%%%%%%%%
\section{The case $2\le a\le e/4$}

\begin{thm}[\cite{Hou-pre}]\label{T3.1}
If $2\le a\le e/4$, $f_{a,q}$ is not a PP of $\f_{q^e}$.
\end{thm}

\begin{proof} We only give a sketch; a detailed proof is given in \cite{Hou-pre}.
Assume to the contrary that $f_{a,q}$ is a PP of $\f_{q^e}$.
Let 
\begin{equation}\label{3.1aa}
F(X,Y)=\frac{f_{a,q}(X)-f_{a,q}(Y)}{X-Y}=\sum_{i=1}^{a}\frac{X^{q^i-2}-Y^{q^i-2}}{X-Y}.
\end{equation}
It can be shown that $F(X,Y)$ is absolutely irreducible, i.e., irreducible over the algebraic closure $\overline\f_q$ of $\f_q$. Let $d=\deg F(X,Y)=q^{a}-3$ and $V_{\f_{q^e}^2}(F)=\{(x,y)\in\f_{q^e}^2:F(x,y)=0\}$. By the Hasse-Weil bound \cite[Theorem~5.28]{Hou-LFF-2018},
\begin{equation}\label{4.1a}
|V_{\f_{q^e}^2}(F)|\ge q^e-(d-1)(d-2)q^{e/2}-\frac 12d(d-1)^2-d-2.
\end{equation}
Let $\lambda$ denote the larger root of 
\[
X^2-(d-1)(d-2)X-\frac 12 d(d-1)^2-d-3\in\Bbb R[X].
\]
We have
\begin{align*}
\lambda\,&=\frac12\Bigl[(d-1)(d-2)+\sqrt{(d-1)^2(d-2)^2+2d(d-1)^2+4d+12}\Bigr]\cr
&=\frac12\Bigl[(d-1)(d-2)+\sqrt{d^4-4d^3+9d^2-6d+16}\,\Bigr]\cr
&<\frac12\bigl[(d-1)(d-2)+d^2\bigr]<d^2<q^{2a}\le q^{e/2}.
\end{align*}
Therefore \eqref{4.1a} gives that $|V_{\f_{q^e}^2}(F(X,Y))|>1$. By \eqref{3.1aa}, $F(X,X)=-2(X^{q-3}+X^{q^2-3}+\cdots+X^{q^{a}-3})=-2X^{-1}f_{a,q}(X)$, which has at most one zero in $\f_{q^e}$ since $f_{a,q}$ is a PP of $\f_{q^e}$. Hence $F(X,Y)$ has a zero $(x,y)\in\f_{q^e}^2$ with $x\ne y$. Consequently, $f_{a,q}$ is not one-to-one on $\f_{q^e}$, which is a contradiction.
\end{proof}

%%%%%%%%%%%%%%%%%%%%%%%%%%%%%%%%%%%%%%%%%%%
%  section 4
%%%%%%%%%%%%%%%%%%%%%%%%%%%%%%%%%%%%%%%%%%%
\section{The case $e/4<a< e/3$}

In this section we will use the following notations. Let $0<r<q$ and $s>0$ be integers such that $s\equiv -7r\pmod{q-1}$ and $s\equiv r/2\pmod q$, i.e., 
$s\equiv r((q^2+1)/2-8q)\pmod{q(q-1)}$; see Remark~\ref{R4.2} for the reason for this requirement. Let 
\begin{equation}\label{4.2}
N=r(q^{e-a}+2q^{e-2a}+4q^{e-3a})+s.
\end{equation}
Write
\begin{gather}\label{4.14}
q^{e-3a}4r=\sum_{j\ge 0}a_jq^j,\\
\label{4.15}
s=\sum_{j\ge 0}s_jq^j,
\end{gather}
where $a_j,s_j\in\{0,\dots,q-1\}$. Let
\begin{gather}\label{4.15.1}
S_1=4rq^{e-2a-1}+8rq^{e-3a-1}+q^{-1}(2s-r),\\ \label{4.15.2} 
S_2=8rq^{e-3a-1}+q^{-1}(2s-r),
\end{gather}
and
\begin{equation}\label{4.16}
\mathcal E_1=\Bigl\{(\epsilon_0,\epsilon_1,\dots):\epsilon_j\ge 0,\ \sum_{j\ge 0}\epsilon_j=|4r|_q+|s|_q,\ \sum_{j\ge 0}\epsilon_jq^j=S_1 \Bigr\},
\end{equation}
\begin{equation}\label{4.17}
\mathcal E_2=\Bigl\{(\epsilon_0,\epsilon_1,\dots):\epsilon_j\ge 0,\ \sum_{j\ge 0}\epsilon_j=|s|_q,\ \sum_{j\ge 0}\epsilon_jq^j=S_2 \Bigr\}.
\end{equation}
The elements of $\mathcal E_1$ can be determined as follows: Write
\begin{equation}\label{4.18}
S_1=\sum_{j\ge 0}e_jq^j,
\end{equation}
where $e_j\in\{0,\dots,q-1\}$. If $\sum_{j\ge 0}e_j>|4r|_q+|s|_q$, $\mathcal E_1=\emptyset$. If $\sum_{j\ge 0}e_j=|4r|_q+|s|_q$, $\mathcal E_1=\{(e_0,e_1,\dots)\}$. If $\sum_{j\ge 0}e_j<|4r|_q+|s|_q$, elements $(\epsilon_0,\epsilon_1,\dots)\in\mathcal E_1$ are obtained from $(e_0,e_1,\dots)$ through borrows in base-$q$. The elements of $\mathcal E_2$ are determined similarly.

\begin{lem}\label{L4.1}
In the above, if $s\le q^{e-2a-1}(q-1)$, then $0<N<q^e-1$ and 
\begin{equation}\label{4.4}
C(N)=\sum_{\alpha_{3i},\alpha_{4i}}\binom{4r}{\alpha_{31},\dots,\alpha_{3a}}\binom{s}{\alpha_{41},\dots,\alpha_{4a}},
\end{equation}
where the sum is over all nonnegative integers $\alpha_{3i},\alpha_{4i}$, $1\le i\le a$, subject to the conditions
\begin{gather}\label{4.5}
\sum_{i=1}^{a}\alpha_{3i}=4r,\\
\label{4.6}
\sum_{i=1}^{a}\alpha_{4i}=s,\\
\label{4.6.1}
\sum_{i=1}^{a}\alpha_{3i}q^{e-3a+i-1}+\sum_{i=1}^{a}\alpha_{4i}q^{i-1}=S_1.
\end{gather}
If $s\le q^{e-3a-1}(q-1)$, then 
\begin{equation}\label{4.6.2}
C(N)=\sum_{\alpha_{4i}}\binom{s}{\alpha_{41},\dots,\alpha_{4a}},
\end{equation}
where the sum is over all integers $\alpha_{4i}$, $1\le i\le a$, subject to \eqref{4.6} and
\begin{equation}\label{4.6.3}
\sum_{i=1}^{a}\alpha_{4i}q^{i-1}=S_2.
\end{equation}
\end{lem}

\begin{proof}
Assume that $s\le q^{e-2a-1}(q-1)$. Obviously, $0<N<q^e-1$. We have
\begin{align}\label{4.7}
f_{a,q}^N&=f_{a,q}^{r(q^{e-a}+2q^{e-2a}+4q^{e-3a})+s}\\ 
&=X^{-2r(q^{e-a}+2q^{e-2a}+4q^{e-3a})-2s}\cr
&\phantom{=}\cdot\Bigl(\sum_{i=1}^{a}X^{q^{e-a+i}}\Bigr)^r\Bigl(\sum_{i=1}^{a}X^{q^{e-2a+i}}\Bigr)^{2r}\Bigl(\sum_{i=1}^{a}X^{q^{e-3a+i}}\Bigr)^{4r}\Bigl(\sum_{i=1}^{a}X^{q^i}\Bigr)^s\cr
&\equiv\sum_{\alpha_{ki}}\binom{r}{\alpha_{11},\dots,\alpha_{1a}}\binom{2r}{\alpha_{21},\dots,\alpha_{2a}}\binom{4r}{\alpha_{31},\dots,\alpha_{3a}}\binom{s}{\alpha_{41},\dots,\alpha_{4a}} X^{E^*}\cr
&\kern 0.8em \pmod{X^{q^e}-X},\nonumber
\end{align}
where
\begin{align}\label{4.8}
E=\,&-2r(q^{e-a}+2q^{e-2a}+4q^{e-3a})-2s+\alpha_{1a}\\
&+\sum_{i=1}^{a-1}\alpha_{1i}q^{e-a+i}+\sum_{i=1}^{a}\alpha_{2i}q^{e-2a+i}+\sum_{i=1}^{a}\alpha_{3i}q^{e-3a+i}+\sum_{i=1}^{a}\alpha_{4i}q^i,\nonumber
\end{align}
and $\alpha_{ki}\ge 0$ are integers such that
\begin{gather}\label{2.4}
\sum_{i=1}^{a}\alpha_{ki}=2^{k-1}r,\quad 1\le k\le 3,\\
\label{2.5}
\sum_{i=1}^{a}\alpha_{4i}=s.
\end{gather}
Assume that $E^*=q^e-1$. We claim that $E=0$.
In fact,
\begin{align*}
E\ge\,&-2r(q^{e-a}+2q^{e-2a}+4q^{e-3a})-2s+r+2rq^{e-2a+1}+4rq^{e-3a+1}+sq\cr
=\,&-2rq^{e-a}+2r(q-2)q^{e-2a}+4r(q-2)q^{e-3a}+s(q-2)+r\cr
>\,&-(q^e-1) \kern 18.5em \text{(since $r\le q^{a}/2$)},
\end{align*}
and
\begin{align*}
E\le\,&-2r(q^{e-a}+2q^{e-2a}+4q^{e-3a})-2s+rq^{e-1}+2rq^{e-a}+4rq^{e-2a}+sq^{a}\cr
=\,&rq^{e-1}-8rq^{e-3a}+s(q^{a}-2)<q^e-1 \kern 1.7em \text{(since $r<q$ and $s\le q^{e-a-1}$)}.
\end{align*}
We further claim that $\alpha_{ka}=2^{k-1}r$ for $k=1,2$, i.e., $\alpha_{ki}=0$ for $k=1,2$ and $1\le i\le a-1$. First, by \eqref{4.8}, $\alpha_{1a}\equiv 2s\equiv r\pmod q$. Since $0\le \alpha_{1a}\le r<q$, we have $\alpha_{1a}=r$. Next, assume to the contrary that $\alpha_{2i}>0$ for some $1\le i\le a-1$. Then by \eqref{4.8},
\begin{align*}
E\,&\le -q^{e-a}+q^{e-a-1}-8rq^{e-3a}+s(q^{a}-2)+r\cr
&<0 \kern 6.5em \text{(since $s\le q^{e-2a-1}(q-1)$)},
\end{align*}
which is a contradiction. Hence the claim is proved. Now the equation $E=0$ is precisely \eqref{4.6.1}. Therefore \eqref{4.4} follows from \eqref{4.7}.

If $s\le q^{e-3a-1}(q-1)$, then by the above argument, $E=0$ implies that $\alpha_{ki}=0$ for $k=1,2,3$ and $1\le i\le a-1$, in which case the equation $E=0$ becomes \eqref{4.6.3}. Therefore \eqref{4.6.2} follows from \eqref{4.7}.
\end{proof}

\begin{rmk}\label{R4.2}\rm 
A necessary condition for \eqref{4.6.1} to have a solution for $\alpha_{3i}$ and $\alpha_{4i}$ ($1\le i\le a$) is that $s\equiv -7r\pmod{q-1}$ and $s\equiv r/2\pmod q$.
\end{rmk}

\begin{prop}\label{P4.2}
If $s\le q^{e-2a-1}(q-1)$, then
\begin{equation}\label{4.19}
C(N)=\sum_{(\epsilon_0,\epsilon_1,\dots)\in\mathcal E_1}\Bigl[\prod_{j\ge 0}(X_j+X_{j+1}+\cdots+X_{j+a-1})^{a_j+s_j}:X_0^{\epsilon_0}X_1^{\epsilon_1}\cdots\Bigr].
\end{equation}
If $s\le q^{e-3a-1}(q-1)$, then 
\begin{equation}\label{4.20}
C(N)=\sum_{(\epsilon_0,\epsilon_1,\dots)\in\mathcal E_2}\Bigl[\prod_{j\ge 0}(X_j+X_{j+1}+\cdots+X_{j+a-1})^{s_j}:X_0^{\epsilon_0}X_1^{\epsilon_1}\cdots\Bigr].
\end{equation}
\end{prop}

\begin{proof}
Let $\alpha_{3i}, \alpha_{4i}$, $1\le i\le a$, be nonnegative integers satisfying \eqref{4.5} -- \eqref{4.6.1}. Write
\begin{gather}\label{4.21}
\alpha_{3i}q^{e-3a}=\sum_{j\ge 0}\alpha_{3i}^{(j)}q^j,\\ \label{4.22}
\alpha_{4i}=\sum_{j\ge 0}\alpha_{4i}^{(j)}q^j,
\end{gather}
where $\alpha_{3i}^{(j)},\alpha_{4i}^{(j)}\in\{0,\dots,q-1\}$. Then
\[
\binom{4r}{\alpha_{31},\dots,\alpha_{3a}}\equiv_p0
\]
unless
\begin{equation}\label{4.22}
\sum_{i=1}^{a}\alpha_{3i}^{(j)}=a_j\quad \text{for all}\ j\ge 0,
\end{equation}
and when \eqref{4.22} is satisfied,
\begin{equation}\label{4.23}
\binom{4r}{\alpha_{31},\dots,\alpha_{3a}}\equiv_p\prod_{j\ge 0}\binom{a_j}{\alpha_{31}^{(j)},\dots,\alpha_{3a}^{(j)}}.
\end{equation}
Similarly, 
\[
\binom{s}{\alpha_{41},\dots,\alpha_{4a}}\equiv_p0
\]
unless 
\begin{equation}\label{4.24} 
\sum_{i=1}^{a}\alpha_{4i}^{(j)}=s_j\quad \text{for all}\ j\ge 0,
\end{equation}
and when \eqref{4.24} is satisfied,
\begin{equation}\label{4.25}
\binom{s}{\alpha_{41},\dots,\alpha_{4a}}\equiv_p\prod_{j\ge 0}\binom{s_j}{\alpha_{41}^{(j)},\dots,\alpha_{4a}^{(j)}}.
\end{equation}
We have
\begin{equation}\label{4.26}
\sum_{i=1}^{a}\alpha_{3i}q^{e-3a+i-1}+\sum_{i=1}^{a}\alpha_{4i}q^{i-1}=\sum_{i=1}^{a}\sum_{j\ge 0}(\alpha_{3i}^{(j)}+\alpha_{4i}^{(j)})q^{i+j-1}=\sum_{k\ge 0}(c_k+d_k)q^k,
\end{equation}
where
\begin{gather}\label{4.27}
c_k=\sum_{i+j=k+1}\alpha_{3i}^{(j)},\\ \label{4.27.1}
d_k=\sum_{i+j=k+1}\alpha_{4i}^{(j)}.
\end{gather}
Assume that \eqref{4.22} and \eqref{4.24} are satisfied. Then
\[
\sum_{k\ge 0}(c_k+d_k)=\sum_{j\ge 0}\sum_{i=1}^a(\alpha_{3i}^{(j)}+\alpha_{4i}^{(j)})=\sum_{j\ge 0}(a_j+s_j)=|4r|_q+|s|_q.
\]
Thus by \eqref{4.16} and \eqref{4.26} we see that \eqref{4.6.1} is equivalent to
\begin{equation}\label{4.28}
(c_0+d_0, c_1+d_1,\dots)\in\mathcal E_1. 
\end{equation}
Now we have
\begin{align*}
C(N)\,&=\sum_{\substack{\alpha_{3i},\alpha_{4i}\cr \eqref{4.5},\eqref{4.6},\eqref{4.6.1}}} \binom{4r}{\alpha_{31},\dots,\alpha_{3a}} \binom{s}{\alpha_{41},\dots,\alpha_{4a}} \kern 6.7em\text{(by \eqref{4.4})}\cr
&=\sum_{\substack{\alpha_{3i}^{(j)},\alpha_{4i}^{(j)}\cr \eqref{4.22},\eqref{4.24},\eqref{4.28}}}\prod_{j\ge0}\binom{a_j}{\alpha_{31}^{(j)},\dots,\alpha_{3a}^{(j)}}\binom{s_j}{\alpha_{41}^{(j)},\dots,\alpha_{4a}^{(j)}}\kern 0.8em\text{(by \eqref{4.23}, \eqref{4.25})}\cr
&=\sum_{(c_0+d_0,c_1+d_1,\dots)\in\mathcal E_1}\biggl(\sum_{\substack{\alpha_{3i}^{(j)}\cr \eqref{4.22},\eqref{4.27}}} \prod_{j\ge0}\binom{a_j}{\alpha_{31}^{(j)},\dots,\alpha_{3a}^{(j)}}\biggr)\cr
&\kern8.5em \cdot\biggl(\sum_{\substack{\alpha_{4i}^{(j)}\cr \eqref{4.24},\eqref{4.27.1}}} \prod_{j\ge0}\binom{s_j}{\alpha_{41}^{(j)},\dots,\alpha_{4a}^{(j)}}\biggr)\cr
&=\sum_{(c_0+d_0,c_1+d_1,\dots)\in\mathcal E_1}\biggl(\Bigl[\prod_{j\ge 0}(X_j+X_{j+1}+\cdots+X_{j+a-1})^{a_j}:X_0^{c_0}X_1^{c_1}\cdots\Bigr]\biggr)\cr
&\kern8.5em \cdot\biggl(\Bigl[\prod_{j\ge 0}(X_j+X_{j+1}+\cdots+X_{j+a-1})^{s_j}:X_0^{d_0}X_1^{d_1}\cdots\Bigr]\biggr)\cr
&\kern25em \text{(by Lemma~\ref{L2.1})}\cr
&=\sum_{(\epsilon_0,\epsilon_1,\dots)\in\mathcal E_1}\Bigl[\prod_{j\ge 0}(X_j+X_{j+1}+\cdots+X_{j+a-1})^{a_j+s_j}:X_0^{\epsilon_0}X_1^{\epsilon_1}\cdots\Bigr],
\end{align*}
which is \eqref{4.19}.

If $s\le q^{e-3a-1}(q-1)$, the proof of \eqref{4.20} is similar but simpler; one uses \eqref{4.6.2} instead of \eqref{4.4}.
\end{proof}

%%%%%%%%%%%%%%%%%%%%%%%%%%%%%%%%%%%%%%%%%%%
\subsection{The case $3a+2\le e<4a$}\
First note that in this case, $a\ge 3$.

\begin{prop}\label{P4.3}
Assume that $3a+2\le e<4a$. Let $u$ be a positive integer such that $30u-15<q$. Let $r=q-2u$ and $s=q-u+(15u-8)q$. Then
\begin{equation}\label{4.31}
C(N)=\binom{-u}{30u-15}\binom{-16u+7}7.
\end{equation}
\end{prop}

\begin{proof}
First note that $0<r<q$ and $0<s<q(q-1)\le q^{e-3a-1}(q-1)$. Clearly,
\begin{equation}\label{4.32}
s_j=\begin{cases}
q-u&\text{if}\ j=0,\cr
15u-8&\text{if}\ j=1,\cr
0&\text{if}\ j>1.
\end{cases}
\end{equation}
We have 
\begin{equation}\label{4.33}
S_2=8rq^{e-3a-1}+q^{-1}(2s-r)=30u-15+(q-16u)q^{e-3a-1}+7q^{e-3a}=\sum_{j\ge 0}e_jq^j, 
\end{equation}
where
\begin{equation}\label{4.34}
e_j=\begin{cases}
30u-15&\text{if}\ j=0,\cr
q-16u&\text{if}\ j=e-3a-1,\cr
7&\text{if}\ j=e-3a,\cr
0&\text{otherwise}.
\end{cases}
\end{equation}
Note that $\sum_{j\ge 0}e_j=q+14u-8=|s|_q$. Hence $\mathcal E_2=\{(e_0,e_1,\dots)\}$. By \eqref{4.20},
\begin{align*}
C(N)
&=\Bigl[\prod_{j\ge 0}(X_j+X_{j+1}+\cdots+X_{j+a-1})^{s_j}:X_0^{e_0}X_1^{e_1}\cdots\Bigr]\cr
&=\Bigl[(X_0+\cdots+X_{a-1})^{q-u}(X_1+\cdots+X_{a})^{15u-8}:X_0^{30u-15}X_{e-3a-1}^{q-16u}X_{e-3a}^7\Bigr]\cr
&=\binom{q-u}{30u-15}\Bigl[(X_1+\cdots+X_{a-1})^{q-16u+7}:X_{e-3a-1}^{q-16u}X_{e-3a}^7\Bigr]\cr
&\kern22.2em \text{(since $a-1\ge e-3a$)}\cr
&=\binom{q-u}{30u-15}\binom{q-16u+7}7\equiv_p\binom{-u}{30u-15}\binom{-16u+7}7.
\end{align*}
\end{proof}

\begin{cor}\label{C4.4}
Assume that $3a+2\le e<4a$ and that $f_{a,q}$ is a PP of $\f_{q^e}$. Then $q\in\{3,3^2,3^3,3^4,5,5^2,7,11,13\}$.
\end{cor}

\begin{proof}
By Hermite's criterion, $C(N)=0$. By \eqref{4.31},
\[
C(N)=\begin{cases}
3^2\cdot 5\cdot 11\cdot 13&\text{if $u=1$ and $q>15$},\cr
2\cdot 3^2\cdot 5^2\cdot 13\cdot 23\cdot 29\cdot 31 &\text{if $u=2$ and $q>45$},\cr
2\cdot 3\cdot 7\cdot 11^2\cdot 19\cdot 23\cdot 41\cdot 43\cdot 47&\text{if $u=3$ and $q>75$},\cr
2^2\cdot 3^4\cdot 19\cdot 29\cdot 31\cdot 53\cdot 59\cdot 61\cdot 107&\text{if $u=4$ and $q>105$},\cr
2\cdot 5\cdot 11\cdot 13\cdot 17\cdot 19\cdot 23\cdot 37\cdot 73\cdot 79\cdot 137\cdot 139&\text{if $u=5$ and $q>135$}.
\end{cases}
\]
The conditions $C(N)=0$ with $1\le u\le 5$ and $q>30u-15$ give the following possibilities for $p$ and $q$.
\[
\begin{tabular}{c|c|c|c}
range of $q$ & $u$ & $p$ & $q$ \\ \hline
$q\le 15$ & & & $3,3^2,5,7,11,13$\\
$15<q\le 45$ & $u=1$ & $3,5,11,13$ & $3^3,5^2$\\
$45<q\le 75$ & $1\le u\le 2$ & $3,5,13$ & none \\
$75<q\le 105$ & $1\le u\le 3$ & $3$ & $3^4$ \\ 
$105<q\le 135$ & $1\le u\le 4$ & $3$ & none \\
$135<q$ & $1\le u\le 5$ & none & none 
\end{tabular}
\]
Hence the claim.
\end{proof}

We now eliminate the $q$'s in Corollary~\ref{C4.4}.

\subsubsection{The case $q=13$}

Let $r=5$ and $s=9+4q$. The base-$q$ digits of $s$ and $S_2$ are given below.
\[
\begin{tabular}{c|ccccc}
digit position & $0$ & $1$ &  \kern5em & &$e-3a$ \\ \hline
$s$ & $9$ & $4$  \\
$S_2$ &$9$ &&&$1$ & $3$
\end{tabular}
\]
By \eqref{4.20},
\[
C(N)=[(X_0+\cdots+X_{a-1})^9(X_1+\cdots+X_{a})^4:X_0^9X_{e-3a-1}X_{e-3a}^3]=\binom 41\not\equiv_{13}0.
\]

\subsubsection{The case $q=11$}

Let $r=3$ and $s=7+2q$. The base-$q$ digits of $s$ and $S_2$ are given below.
\[
\begin{tabular}{c|ccccc}
digit position & $0$ & $1$  & \kern5em & &$e-3a$\\ \hline
$s$ & $7$ & $2$  \\
$S_2$ &$5$ &&& $2$ & $2$
\end{tabular}
\]
We have
\begin{align*}
C(N)&=[(X_0+\cdots+X_{a-1})^7(X_1+\cdots+X_{a})^2:X_0^5X_{e-3a-1}^2X_{e-3a}^2]\cr
&=\binom 75[(X_1+\cdots+X_{a-1})^4:X_{e-3a-1}^2X_{e-3a}^2]=\binom 75\binom 42\not\equiv_{11}0.
\end{align*}

\subsubsection{The case $q=7$}

Let $r=1$ and $s=4+q$. The base-$q$ digits of $s$ and $S_2$ are given below.
\[
\begin{tabular}{c|ccccc}
digit position & $0$ & $1$ & \kern5em & &$e-3a$ \\ \hline
$s$ & $4$ & $1$ \\
$S_2$ &$3$ &&& $1$ & $1$
\end{tabular}
\]
We have
\[
C(N)=[(X_0+\cdots+X_{a-1})^4(X_1+\cdots+X_{a}):X_0^3X_{e-3a-1}X_{e-3a}]=\binom 43\binom 21\not\equiv_70.
\]

\subsubsection{The case $q=5^2$}

First assume that $3a+3\le e<4a$. Let $r=3$ and $s=14+12q+q^{e-3a}$. The base-$q$ digits of $s$, $q^{e-3a}4r$ and $S_1$ are given below.
\[
\begin{tabular}{c|ccccccc}
digit position & $0$ & $1$ & \kern4em & & $e-3a$ & \kern4em & $e-2a-1$ \\ \hline
$s$ & $14$ & $12$ &&& $1$\\
$q^{e-3a}4r$ &&&&& $12$ \\
$S_1$ & $0$ &$1$ && $1$ & $1$ && $12$
\end{tabular}
\]
Let
\[
F=(X_0+\cdots+X_{a-1})^{14}(X_1+\cdots+X_{a})^{12}(X_{e-3a}+\cdots+X_{e-2a-1})^{13}.
\]
When $e\ge 3a+4$, we have $\mathcal E_1=\{(\epsilon_{i0},\epsilon_{i1},\dots):1\le i\le 4\}$, where
\[
\begin{array}{rcccccccccl}
&\scriptstyle 0 &\scriptstyle 1 &\kern4em &&&\scriptstyle e-3a&\kern4em & &\scriptstyle e-2a-1 \cr   
(\epsilon_{10},\epsilon_{11},\dots)=(\kern-0.5em&25 &&&&1&1&&&12&\kern-1em), \cr
(\epsilon_{20},\epsilon_{21},\dots)=(\kern-0.5em&&1&&25&&1&&&12&\kern-1em), \cr
(\epsilon_{30},\epsilon_{31},\dots)=(\kern-0.5em&&1 &&&26&&&&12&\kern-1em), \cr
(\epsilon_{40},\epsilon_{41},\dots)=(\kern-0.5em&&1 &&&1&1&&25&11&\kern-1em).
\end{array}
\]
In the above, if $e=3a+3$, then 
\[
\begin{array}{rcccccccl}
&\scriptstyle 0 &\scriptstyle 1 &\scriptstyle 2&\scriptstyle 3&\kern4em & &\scriptstyle e-2a-1 \cr
(\epsilon_{20},\epsilon_{21},\dots)=(\kern-0.5em&&26&&1&&&12&\kern-1em). \end{array}
\]
In both cases, we find that
\[
[F:X_0^{\epsilon_{i0}}X_1^{\epsilon_{i1}}\cdots]=
\begin{cases}
0&\text{if}\ i=1,3,4,\cr
3&\text{if}\ i=2.
\end{cases}
\]
By \eqref{4.19},
\[
C(N)=\sum_{i=1}^4[F:X_0^{\epsilon_{i0}}X_1^{\epsilon_{i1}}\cdots]=3\not\equiv_5 0.
\]

Now assume that $e=3a+2$. Let $r=3$ and $s=14+13q$. The base-$q$ digits of $s$ and $S_2$ are given below.
\[
\begin{tabular}{c|ccc}
digit position & $0$ & $1$ & $2$\\ \hline
$s$ & $14$ & $13$ \\
$S_2$ &$2$ && $1$ 
\end{tabular}
\]
We have $\mathcal E_2=\{(2,25)\}$. Hence
\[
C(N)=[(X_0+\cdots+X_{a-1})^{14}(X_1+\cdots+X_{a})^{13}:X_0^2X_1^{25}]=\binom{14}2\not\equiv_5 0.
\]

\subsubsection{The case $q=5$}

First assume that $3a+3\le e<4a$. Let $r=3$ and $s=4+2q+q^{e-3a}$. The base-$q$ digits of $s$, $q^{e-3a}4r$ and $S_1$ are given below.
\[
\begin{tabular}{c|ccccccccc}
digit position & $0$ & $1$ & \kern4em & & $e-3a$ & &\kern4em & $e-2a-1$ \\ \hline
$s$ & $4$ & $2$ &&& $1$\\
$q^{e-3a}4r$ &&&&& $2$ & $2$\\
$S_1$ & $0$ &$1$ && $1$ &$0$&$1$ && $2$ & $2$
\end{tabular}
\]
Let
\begin{align*}
F=\,&(X_0+\cdots+X_{a-1})^4(X_1+\cdots+X_{a})^2(X_{e-3a}+\cdots+X_{e-2a-1})^3\cr
&\cdot (X_{e-3a+1}+\cdots+X_{e-2a})^2.
\end{align*}
When $e\ge 3a+4$ and $a\ge 4$, we have $\mathcal E_1=\{(\epsilon_{i0},\epsilon_{i1},\dots):1\le i\le 5\}$, where
\[
\begin{array}{rcccccccccccl}
&\scriptstyle 0 &\scriptstyle 1 &\kern4em &&&\scriptstyle e-3a&&\kern4em & &\scriptstyle e-2a-1 \cr   
(\epsilon_{10},\epsilon_{11},\dots)=(\kern-0.5em&5&0&&&1&&1&&&2&2&\kern-0.5em), \cr
(\epsilon_{20},\epsilon_{21},\dots)=(\kern-0.5em&&1&&5&0&&1&&&2&2&\kern-0.5em), \cr
(\epsilon_{30},\epsilon_{31},\dots)=(\kern-0.5em&&1 &&&1&5&&&&2&2&\kern-0.5em), \cr
(\epsilon_{40},\epsilon_{41},\dots)=(\kern-0.5em&&1 &&&1&&1&&5&1&2&\kern-0.5em),\cr
(\epsilon_{50},\epsilon_{51},\dots)=(\kern-0.5em&&1 &&&1&&1&&&7&1&\kern-0.5em).
\end{array}
\]
In the above, if $e=3a+3$, then
\[
\begin{array}{rcccccccccl}
&\scriptstyle 0 &\scriptstyle 1 &\scriptstyle 2&\scriptstyle 3&&\kern4em & &\scriptstyle e-2a-1 \cr   
(\epsilon_{20},\epsilon_{21},\dots)=(\kern-0.5em&&6&0&&1&&&2&2&\kern-0.5em);
\end{array}
\]
if $a=3$, then 
\[
\begin{array}{rccccccccl}
&\scriptstyle 0 &\scriptstyle 1 &\kern4em &&\scriptstyle e-3a& &\scriptstyle e-2a-1 \cr
(\epsilon_{40},\epsilon_{41},\dots)=(\kern-0.5em&&1 &&1&&6&1&2&\kern-0.5em).
\end{array}
\]
In all these cases, we find that
\[
[F:X_0^{\epsilon_{i0}}X_1^{\epsilon_{i1}}\cdots]=
\begin{cases}
0&\text{if}\ i=1,3,4,5,\cr
3&\text{if}\ i=2.
\end{cases}
\]
Hence
\[ 
C(N)=\sum_{i=1}^5[F:X_0^{\epsilon_{i0}}X_1^{\epsilon_{i1}}\cdots]=3\not\equiv_5 0.
\]

Now assume that $e=3a+2$. Let $r=3$ and $s=4+3q$. The base-$q$ digits of $s$ and $S_2$ are given below.
\[
\begin{tabular}{c|cccc}
digit position & $0$ & $1$ & $2$ &3\\ \hline
$s$ & $4$ & $3$ \\
$S_2$ & $2$ & $0$ &$0$ & $1$ 
\end{tabular}
\]
We have $\mathcal E_2=\{(2,0,5)\}$. Hence
\[
C(N)=[(X_0+\cdots+X_{a-1})^4(X_1+\cdots+X_{a})^3:X_0^2X_2^5]=\binom42\not\equiv_5 0.
\]

\subsubsection{The case $q=3^4$}

Let $r=q-10$ and $s=(q-5)+67q$. First assume that $3a+3\le e<4a$. The base-$q$ digits of $s$ and $S_2$ are given below.
\[
\begin{tabular}{c|ccccc}
digit position & $0$ & $1$ & \kern4em &$e-3a-1$ &  \\ \hline
$s$ & $76$ & $67$ \\
$S_2$ &$54$ & $1$ &&$1$ &$7$  
\end{tabular}
\]
Let 
\[
F=(X_0+\cdots+X_{a-1})^{76}(X_1+\cdots+X_{a})^{67}.
\]
When $e\ge 3a+4$, we have $\mathcal E_2=\{(\epsilon_{i0},\epsilon_{i1},\dots):1\le i\le 3\}$, where
\[
\begin{array}{rccccccl}
&\scriptstyle 0 &\scriptstyle 1 &\kern4em &&&\scriptstyle e-3a \cr   
(\epsilon_{10},\epsilon_{11},\dots)=(\kern-0.5em&135&0&&&1&7&\kern-0.5em), \cr
(\epsilon_{20},\epsilon_{21},\dots)=(\kern-0.5em&54&1&&81&0&7&\kern-0.5em), \cr
(\epsilon_{30},\epsilon_{31},\dots)=(\kern-0.5em&54&1 &&&82&6&\kern-0.5em).
\end{array}
\]
In the above, if $e=3a+3$, then
\[
\begin{array}{rccccl}   
(\epsilon_{20},\epsilon_{21},\dots)=(\kern-0.5em&54&82&0&7&\kern-0.5em)
\end{array}.
\] 
In both cases, we find that
\[ 
[F:X_0^{\epsilon_{i0}}X_1^{\epsilon_{i1}}\cdots]=
\begin{cases}
0&\text{if}\ i=1,\cr
2&\text{if}\ i=2,3.
\end{cases}
\]
Hence $C(N)=\sum_{i=1}^3[F:X_0^{\epsilon_{i0}}X_1^{\epsilon_{i1}}\cdots]=4\not\equiv_3 0$.

Now assume that $e=3a+2$. In this case, the digits of $S_2$ are modified as follows.
\[
\begin{tabular}{c|ccc}
digit position & $0$ & $1$  &2  \\ \hline
$s$ & $76$ & $67$ \\
$S_2$ &$54$ & $2$ &$7$  
\end{tabular}
\]
We have $\mathcal E_2=\{(\epsilon_{i0},\epsilon_{i1},\dots):i=1,2\}$, where 
\[
\begin{array}{rcccl}  
(\epsilon_{10},\epsilon_{11},\dots)=(\kern-0.5em&135&1&7&\kern-0.5em), \cr
(\epsilon_{20},\epsilon_{21},\dots)=(\kern-0.5em&54&83&6&\kern-0.5em).
\end{array}
\]
It is easy to see that 
\[ 
[F:X_0^{\epsilon_{i0}}X_1^{\epsilon_{i1}}\cdots]=
\begin{cases}
0&\text{if}\ i=1,\cr
1&\text{if}\ i=2.
\end{cases}
\]
Hence $C(N)=\sum_{i=1}^2[F:X_0^{\epsilon_{i0}}X_1^{\epsilon_{i1}}\cdots]=1\not\equiv_3 0$.

\subsubsection{The case $q=3^3$}
Let $r=1$ and $s=14+5q$. The base-$q$ digits of $s$ and $S_2$ are given below.
\[
\begin{tabular}{c|ccccc}
digit position & $0$ & $1$ &\kern4em & &$e-3a$  \\ \hline
$s$ & $14$ & $5$ \\
$S_2$ &$11$ & &&$8$  
\end{tabular}
\]
We have
\[
C(N)=[(X_0+\cdots+X_{a-1})^{14}(X_1+\cdots+X_{a})^5:X_0^{11}X_{e-3a-1}^8]=\binom{14}{11}\not\equiv_3 0.
\]

\subsubsection{The case $q=3^2$}

First assume that $3a+3\le e<4a$. Let $r=1$ and $s=5+2q+q^2+q^{e-3a}$. The base-$q$ digits of $s$, $q^{e-3a}4r$ and $S_1$ are given below.
\[
\begin{tabular}{c|cccccccc}
digit position & $0$ & $1$ && \kern4em & & $e-3a$ & \kern4em & $e-2a-1$ \\ \hline
$s$ & $5$ & $2$ &$1$&&& $1$\\
$q^{e-3a}4r$ &&&&&& $4$ \\
$S_1$ & $5$ &$2$ &&& $1$ &$1$ && $4$
\end{tabular}
\]
We have
\begin{align*}
&C(N)\cr
&=[(X_0+\cdots+X_{a-1})^5(X_1+\cdots+X_{a})^2(X_2+\cdots+X_{a+1})(X_{e-3a}+\cdots+X_{e-2a-1})^5\cr
&\kern1em:X_0^5X_1^2X_{e-3a-1}X_{e-3a}X_{e-2a-1}^4]\cr
&=\binom55\binom22\binom54[(X_2+\cdots+X_{a+1}):X_{e-3a-1}]=5\not\equiv_3 0.
\end{align*}

Now assume that $e=3a+2$. Let $r=1$ and $s=5+4q$. The base-$q$ digits of $s$ and $S_2$ are given below.
\[
\begin{tabular}{c|ccc}
digit position & $0$ & $1$ & 2\\ \hline
$s$ & 5 & 4\\
$S_2$ & &&1
\end{tabular}
\]
We have $\mathcal E_2=\{(0,9)\}$ and
\[
C(N)=[(X_0+\cdots+X_{a-1})^5(X_1+\cdots+X_{a})^4:X_1^9]=1\not\equiv_3 0.
\]

\subsubsection{The case $q=3$}

Let $r=1$ and $s=2+q^{e-3a}$.
The base-$q$ digits of $s$, $q^{e-3a}4r$ and $S_1$ are given below. 
\[
\begin{tabular}{c|cccccccccc}
digit position & $0$ & $1$ & &\kern4em & & $e-3a$ && \kern4em & $e-2a-1$ \\ \hline
$s$ & $2$ &0 &&&& $1$\\
$q^{e-3a}4r$ &&&&&& $1$ &1\\
$S_1$ & 1 & &&& $1$ &0 &1& &$1$ &1
\end{tabular}
\]
We have
\begin{align*}
&C(N)\cr
&=[(X_0+\cdots+X_{a-1})^2(X_{e-3a}+\cdots+X_{e-2a-1})^2(X_{e-3a+1}+\cdots+X_{e-2a})\cr
&\kern1em :X_0X_{e-3a-1}X_{e-3a+1}X_{e-2a-1}X_{e-2a}]\cr
&=\binom21 [(X_1+\cdots+X_{a-1})(X_{e-3a}+\cdots+X_{e-2a-1})^2:X_{e-3a-1}X_{e-3a+1}X_{e-2a-1}]\cr
&=\binom21\binom21 [(X_1+\cdots+X_{a-1}):X_{e-3a-1}]=4\not\equiv_3 0.
\end{align*}

%%%%%%%%%%%%%%%%%%%%%%%%%%%%%%%%%%%%%%%%%%%
\subsection{The case $e=3a+1$}

\subsubsection{The case $q>8$}
Let $r=q-2$ and $s=(q-1)+7q$.  The base-$q$ digits of $s$, $q^{e-3a}4r$ and $S_1$ are given below.
\[
\begin{tabular}{c|cccccc}
digit position & $0$ & 1 &&\kern4em & $e-2a-1$ \\ \hline
$s$ & $q-1$ & $7$\\
$q^{e-3a}4r$ && $q-8$ &3\\
$S_1$ & $q-1$ & 7&&& $q-8$ &3
\end{tabular}
\]
We have
\begin{align*}
C(N)\,&=[(X_0+\cdots+X_{a-1})^{q-1}(X_1+\cdots+X_{a})^{q-1}(X_2+\cdots+X_{a+1})^3\cr
&\kern1em :X_0^{q-1}X_1^7X_{a}^{q-8}X_{a+1}^3]\cr
&=\binom{q-1}7 [(X_2+\cdots+X_{a})^{q-8}(X_2+\cdots+X_{a+1})^3:X_{a}^{q-8}X_{a+1}^3]\cr
&\equiv_p -1\not\equiv_p 0.
\end{align*}

\subsubsection{The case $q=7$}

Let $r=1$ and $s=4+q$.  The base-$q$ digits of $s$, $q^{e-3a}4r$ and $S_1$ are given below.
\[
\begin{tabular}{c|cccc}
digit position & $0$ & 1 &\kern4em & $e-2a-1$ \\ \hline
$s$ & $4$ & $1$\\
$q^{e-3a}4r$ &&4\\
$S_1$ & $4$ & 1&&4
\end{tabular}
\]
We have
\[
C(N)=[(X_0+\cdots+X_{a-1})^4(X_1+\cdots+X_{a})^5:X_0^4X_1X_{a}^4]=\binom51 \not\equiv_7 0.
\]

\subsubsection{The case $q=5$}

Let $r=2$ and $s=1+q^2$.  First assume that $a\ge 3$. The base-$q$ digits of $s$, $q^{e-3a}4r$ and $S_1$ are given below.
\[
\begin{tabular}{c|cccccc}
digit position & $0$ & 1&2 &\kern4em& $e-2a-1$ \\ \hline
$s$ & $1$ & & $1$\\
$q^{e-3a}4r$ &&3&1\\
$S_1$ &1& 0&$1$ & &3&1
\end{tabular}
\]
Let
\[
F=(X_0+\cdots+X_{a-1})(X_1+\cdots+X_{a})^3(X_2+\cdots+X_{a+1})^2.
\]
We find that $C(N)=[F:X_0X_2X_{a}^3X_{a+1}]=3\not\equiv_5 0$.

When $a=2$, the digits of $S_1$ are modified as follows.
\[
\begin{tabular}{c|cccc}
digit position & $0$ & 1& 2 &3 \\ \hline
$s$ & $1$ & & $1$\\
$q^{e-3a}4r$ &&3&1\\
$S_1$ &1& 0&$4$ &1
\end{tabular}
\]
We find that $C(N)=[F:X_0X_2^4X_3]=2\not\equiv_5 0$.

\subsubsection{The case $q=3$}

Let $r=1$ and $s=2+q$.  First assume that $a\ge 3$. The base-$q$ digits of $s$, $q^{e-3a}4r$ and $S_1$ are given below.
\[
\begin{tabular}{c|cccccc}
digit position & $0$ &1 & &\kern4em& $e-2a-1$ \\ \hline
$s$ & $2$ & $1$\\
$q^{e-3a}4r$ &&1&1\\
$S_1$ &2& 0&$1$ & &1&1
\end{tabular}
\]
Let
\[
F=(X_0+\cdots+X_{a-1})^2(X_1+\cdots+X_{a})^2(X_2+\cdots+X_{a+1}).
\]
We find that $C(N)=[F:X_0^2X_2X_{a}X_{a+1}]=2\not\equiv_3 0$.

When $a=2$, the digits of $S_1$ is modified as follows. 
\[
\begin{tabular}{c|cccc}
digit position & $0$ &$1$ & 2&3 \\ \hline
$s$ & $2$ & $1$\\
$q^{e-3a}4r$ &&1&1\\
$S_1$ &2& 0&$2$ &1
\end{tabular}
\]
We find that $C(N)=[F:X_0^2X_2^2X_3]=1\not\equiv_3 0$.

%%%%%%%%%%%%%%%%%%%%%%%%%%%%%%%%%%%%%%%%%%%
%  section 5
%%%%%%%%%%%%%%%%%%%%%%%%%%%%%%%%%%%%%%%%%%%
\section{The case $e/3<a< e/2$}

The proof in this section is smilier to but simpler than Section 4. We will use the following notations. Let $0<r<q$ and $s>0$ be integers such that $s\equiv -3r\pmod{q-1}$ and $s\equiv r/2\pmod q$, i.e., $s\equiv r((q^2+1)/2-4q)\pmod{q(q-1)}$. Let 
\begin{equation}\label{5.2}
N=r(q^{e-a}+2q^{e-2a})+s.
\end{equation}
Write
\begin{gather}\label{5.14}
q^{e-2a}2r=\sum_{j\ge 0}a_jq^j,\\
\label{5.15}
s=\sum_{j\ge 0}s_jq^j,
\end{gather}
where $a_j,s_j\in\{0,\dots,q-1\}$. Let
\begin{gather}\label{5.15.1}
S_1=2rq^{e-a-1}+4rq^{e-2a-1}+q^{-1}(2s-r),\\ \label{5.15.2} 
S_2=4rq^{e-2a-1}+q^{-1}(2s-r),
\end{gather}
and
\begin{equation}\label{5.16}
\mathcal E_1=\Bigl\{(\epsilon_0,\epsilon_1,\dots):\epsilon_j\ge 0,\ \sum_{j\ge 0}\epsilon_j=|2r|_q+|s|_q,\ \sum_{j\ge 0}\epsilon_jq^j=S_1 \Bigr\},
\end{equation}
\begin{equation}\label{5.17}
\mathcal E_2=\Bigl\{(\epsilon_0,\epsilon_1,\dots):\epsilon_j\ge 0,\ \sum_{j\ge 0}\epsilon_j=|s|_q,\ \sum_{j\ge 0}\epsilon_jq^j=S_2 \Bigr\}.
\end{equation}
The elements of $\mathcal E_1$ can be determined as follows: Write
\begin{equation}\label{5.18}
S_1=\sum_{j\ge 0}e_jq^j,
\end{equation}
where $e_j\in\{0,\dots,q-1\}$. If $\sum_{j\ge 0}e_j>|2r|_q+|s|_q$, $\mathcal E_1=\emptyset$. If $\sum_{j\ge 0}e_j=|2r|_q+|s|_q$, $\mathcal E_1=\{(e_0,e_1,\dots)\}$. If $\sum_{j\ge 0}e_j<|2r|_q+|s|_q$, elements $(\epsilon_0,\epsilon_1,\dots)\in\mathcal E_1$ are obtained from $(e_0,e_1,\dots)$ through borrows in base-$q$. The elements of $\mathcal E_2$ are determined similarly.

\begin{lem}\label{L5.1}
If $s\le q^{e-a-1}(q-1)$, then $0<N<q^e-1$ and
\begin{equation}\label{5.4}
C(N)=\sum_{\alpha_{2i},\alpha_{3i}}\binom{2r}{\alpha_{21},\dots,\alpha_{2a}}\binom{s}{\alpha_{31},\dots,\alpha_{3a}},
\end{equation}
where the sum is over all nonnegative integers $\alpha_{2i},\alpha_{3i}$, $1\le i\le a$, subject to the conditions
\begin{gather}\label{5.5}
\sum_{i=1}^{a}\alpha_{2i}=2r,\\
\label{5.6}
\sum_{i=1}^{a}\alpha_{3i}=s,\\
\label{5.6.1}
\sum_{i=1}^{a}\alpha_{2i}q^{e-2a+i-1}+\sum_{i=1}^{a}\alpha_{3i}q^{i-1}=S_1.
\end{gather}
If $s\le q^{e-2a-1}(q-1)$, then 
\begin{equation}\label{5.6.2}
C(N)=\sum_{\alpha_{3i}}\binom{s}{\alpha_{31},\dots,\alpha_{3a}},
\end{equation}
where the sum is over all integers $\alpha_{3i}$, $1\le i\le a$, subject to \eqref{5.6} and
\begin{equation}\label{5.6.3}
\sum_{i=1}^{a}\alpha_{3i}q^{i-1}=S_2.
\end{equation}
\end{lem}

\begin{proof}
The proof is almost identical to the proof of Lemma~\ref{L4.1}. Assume that $s\le q^{e-a-1}(q-1)$. Clearly, $0<N<q^e-1$. We have
\begin{align}\label{5.7}
f_{a,q}^N&=X^{-2r(q^{e-a}+2q^{e-2a})-2s}\Bigl(\sum_{i=1}^{a}X^{q^{e-a+i}}\Bigr)^r\Bigl(\sum_{i=1}^{a}X^{q^{e-2a+i}}\Bigr)^{2r}\Bigl(\sum_{i=1}^{a}X^{q^i}\Bigr)^s\\
&\equiv\sum_{\alpha_{ki}}\binom{r}{\alpha_{11},\dots,\alpha_{1a}}\binom{2r}{\alpha_{21},\dots,\alpha_{2a}}\binom{s}{\alpha_{31},\dots,\alpha_{3a}} X^{E^*}\pmod{X^{q^e}-X},\nonumber
\end{align}
where
\begin{equation}\label{5.8}
E=-2r(q^{e-a}+2q^{e-2a})-2s+\alpha_{1a}+\sum_{i=1}^{a-1}\alpha_{1i}q^{e-a+i}+\sum_{i=1}^{a}\alpha_{2i}q^{e-2a+i}+\sum_{i=1}^{a}\alpha_{3i}q^i,
\end{equation}
and $\alpha_{ki}\ge 0$ are integers such that
\begin{gather}\label{5.9}
\sum_{i=1}^{a}\alpha_{ki}=2^{k-1}r,\quad  k=1,2,\\
\label{5.10}
\sum_{i=1}^{a}\alpha_{3i}=s.
\end{gather}
Assume that $E^*=q^e-1$. We claim that $E=0$. First,
\begin{align*}
E\ge\,&-2r(q^{e-a}+2q^{e-2a})-2s+r+2rq^{e-2a+1}+sq\cr
=\,&-2rq^{e-a}+2r(q-2)q^{e-2a}+s(q-2)+r\cr
>\,&-(q^e-1) \kern 8.7em \text{(since $r\le q^{a}/2$)}.
\end{align*}
Write $E=k(q^e-1)$, where $k\ge 0$. Then by \eqref{5.8}, $-k\equiv -r+\alpha_{1a}\pmod q$, and hence $\alpha_{1a}\ge r-k$. Thus by \eqref{5.8},
\begin{align*}
k(q^e-1)\,&=E\le -2r(q^{e-a}+2q^{e-2a})-2s+r-k+kq^{e-1}+2rq^{e-a}+sq^{a}\cr
&<k(q^{e-1}-1)+sq^a,
\end{align*}
whence $k<sq^a/(q^e-q^{e-1})\le 1$. Hence the claim is proved.
By \eqref{5.8}, $\alpha_{1a}\equiv 2s\equiv r\pmod q$. Since $0\le \alpha_{1a}\le r<q$, we have $\alpha_{1a}=r$. Now the equation $E=0$ becomes \eqref{5.6.1}. Therefore \eqref{5.4} follows from \eqref{5.7}.

If $s\le q^{e-2a-1}(q-1)$, we claim that $\alpha_{2a}=2r$, i.e., $\alpha_{2i}=0$ for $1\le i\le a-1$.  Assume to the contrary that $\alpha_{2i}>0$ for some $1\le i\le a-1$. Then by \eqref{5.8},
\begin{align*}
E\,&\le -q^{e-a}+q^{e-a-1}-4rq^{e-2a}+s(q^{a}-2)+r\cr
&<0 \kern 6.5em \text{(since $s\le q^{e-2a-1}(q-1)$)},
\end{align*}
which is a contradiction. Hence the claim is proved.  
Now the equation $E=0$ becomes \eqref{5.6.3}. Therefore \eqref{5.6.2} follows from \eqref{5.7}.
\end{proof}

\begin{prop}\label{P5.2}
If $s\le q^{e-a-1}(q-1)$, then
\begin{equation}\label{5.19}
C(N)=\sum_{(\epsilon_0,\epsilon_1,\dots)\in\mathcal E_1}\Bigl[\prod_{j\ge 0}(X_j+X_{j+1}+\cdots+X_{j+a-1})^{a_j+s_j}:X_0^{\epsilon_0}X_1^{\epsilon_1}\cdots\Bigr].
\end{equation}
If $s\le q^{e-2a-1}(q-1)$, then 
\begin{equation}\label{5.20}
C(N)=\sum_{(\epsilon_0,\epsilon_1,\dots)\in\mathcal E_2}\Bigl[\prod_{j\ge 0}(X_j+X_{j+1}+\cdots+X_{j+a-1})^{s_j}:X_0^{\epsilon_0}X_1^{\epsilon_1}\cdots\Bigr],
\end{equation}
\end{prop}

\begin{proof}
Identical to the proof of Proposition~\ref{P4.2}.
\end{proof}

%%%%%%%%%%%%%%%%%%%%%%%%%%%%%%%%%%%%%%%%%%%
\subsection{The case $2a+2\le e<3a$}

Note that in this case, $a\ge 3$.

\begin{prop}\label{P5.3}
Assume that $2a+2\le e<3a$. Let $u$ be a positive integer such that $14u-7<q$. Let $r=q-2u$ and $s=q-u+(7u-4)q$. Then
\begin{equation}\label{5.31}
C(N)=\binom{-u}{14u-7}\binom{-8u+3}3.
\end{equation}
\end{prop}

\begin{proof}
First note that $0<r<q$ and $0<s<q(q-1)\le q^{e-2a-1}(q-1)$. Clearly,
\begin{equation}\label{5.32}
s_j=\begin{cases}
q-u&\text{if}\ j=0,\cr
7u-4&\text{if}\ j=1,\cr
0&\text{if}\ j>1.
\end{cases}
\end{equation}
We have 
\begin{equation}\label{5.33}
S_2=4rq^{e-2a-1}+q^{-1}(2s-r)=14u-7+(q-8u)q^{e-2a-1}+3q^{e-2a}=\sum_{j\ge 0}e_jq^j, 
\end{equation}
where
\begin{equation}\label{5.34}
e_j=\begin{cases}
14u-7&\text{if}\ j=0,\cr
q-8u&\text{if}\ j=e-2a-1,\cr
3&\text{if}\ j=e-2a,\cr
0&\text{otherwise}.
\end{cases}
\end{equation}
Since $\sum_{j\ge 0}e_j=q+6u-4=|s|_q$, we have $\mathcal E_2=\{(e_0,e_1,\dots)\}$. By \eqref{5.20},
\begin{align*}
C(N)
&=\Bigl[\prod_{j\ge 0}(X_j+X_{j+1}+\cdots+X_{j+a-1})^{s_j}:X_0^{e_0}X_1^{e_1}\cdots\Bigr]\cr
&=\Bigl[(X_0+\cdots+X_{a-1})^{q-u}(X_1+\cdots+X_{a})^{7u-4}:X_0^{14u-7}X_{e-2a-1}^{q-8u}X_{e-2a}^3\Bigr]\cr
&=\binom{q-u}{14u-7}\Bigl[(X_1+\cdots+X_{a-1})^{q-8u+3}:X_{e-2a-1}^{q-8u}X_{e-2a}^3\Bigr]\cr
&\kern21.5em \text{(since $a-1\ge e-2a$)}\cr
&=\binom{q-u}{14u-7}\binom{q-8u+3}3\equiv_p\binom{-u}{14u-7}\binom{-8u+3}3.
\end{align*}
\end{proof}

\begin{cor}\label{C5.4}
Assume that $2a+2\le e<3a$ and that $f_{a,q}$ is a PP of $\f_{q^e}$. Then $q\in\{3,5,5^2,7,7^2\}$.
\end{cor}

\begin{proof}
By Hermite's criterion, $C(N)=0$. By \eqref{5.31},
\[
C(N)=\begin{cases}
5\cdot 7&\text{if $u=1$ and $q>7$},\cr
2\cdot 5\cdot 7\cdot 11\cdot 13 &\text{if $u=2$ and $q>21$},\cr
2\cdot 3^2\cdot 7\cdot 11\cdot 23\cdot 37&\text{if $u=3$ and $q>35$},\cr
2^2\cdot 5^3\cdot 13\cdot 17\cdot 29\cdot 31&\text{if $u=4$ and $q>49$}.
\end{cases}
\]
The conditions $C(N)=0$ with $1\le u\le 4$ and $q>14u-7$ give the following possibilities for $p$ and $q$.
\[
\begin{tabular}{c|c|c|c}
range of $q$ & $u$ & $p$ & $q$ \\ \hline
$q\le 7$ & & & $3,5,7$\\
$7<q\le 21$ & $u=1$ & $5,7$ & none\\
$21<q\le 35$ & $1\le u\le 2$ & $5,7$ & $5^2$ \\
$35<q\le 49$ & $1\le u\le 3$ & $7$ & $7^2$ \\ 
$49<q$ & $1\le u\le 4$ & none & none \\ 
\end{tabular}
\]
Hence the claim.
\end{proof}

\subsubsection{The case $q=7^2$}

First assume that $2a+3\le e<3a$. Let $r=6$ and $s=3+25q+2q^2$. The base-$q$ digits of $s$ and $S_2$ are given below.
\[
\begin{tabular}{c|cccccc}
digit position & $0$ & $1$ & $2$ &\kern4em& &$e-2a$ \\ \hline
$s$ & $3$ & $25$ & $2$\\
$S_2$ &1& 5& & &24
\end{tabular}
\]
We find that
\begin{align*}
C(N)\,&=[(X_0+\cdots+X_{a-1})^3(X_1+\cdots+X_{a})^{25}(X_2+\cdots+X_{a+1})^2:X_0X_1^5X_{e-2a-1}^{24}]\cr
&=\binom 31\binom{27}5\not\equiv_7 0.
\end{align*}

Now assume that $e=2a+2$. Let $r=6$ and $s=3+26q+q^3$. The base-$q$ digits of $s$, $q^{e-2a}2r$ and $S_1$ are given below.
\[
\begin{tabular}{c|cccccc}
digit position &0 &1& 2&3&\kern4em& $e-a-1$ \\ \hline
$s$ & 3 & 26 &0&1\\
$q^{e-2a}2r$ &&&12\\
$S_1$ &3& 25&2 & & &12
\end{tabular}
\]
We find that
\begin{align*}
C(N)
&=[(X_0+\cdots+X_{a-1})^3(X_1+\cdots+X_{a})^{26}(X_2+\cdots+X_{a+1})^{12}(X_3+\cdots+X_{a+2})\cr
&\kern1em :X_0^3X_1^{25}X_2^2X_{a+1}^{12}]\cr
&=\binom{26}{25}\binom{12}{11}\not\equiv_7 0.
\end{align*}

\subsubsection{The case $q=7$}

Let $r=1$ and $s=4+5q$. First assume that $2a+3\le e<3a$. The base-$q$ digits of $s$ and $S_2$ are given below.
\[
\begin{tabular}{c|cccc}
digit position &0 &1& \kern4em& $e-2a-1$ \\ \hline
$s$ &  4& 5 \\
$S_2$ &4& 1& & 4
\end{tabular}
\]
We have
\[
C(N)=[(X_0+\cdots+X_{a-1})^4(X_1+\cdots+X_{a})^5:X_0^4X_1X_{e-2a-1}^4]=\binom51 \not\equiv_7 0.
\]

Now assume that $e=2a+2$. The digits of $S_2$ are modified as follows. 
\[
\begin{tabular}{c|cc}
digit position &0 &1 \\ \hline
$s$ &  4& 5 \\
$S_2$ &4& 5
\end{tabular}
\]
In this case, $C(N)=1\not\equiv_7 0$. 

\subsubsection{The case $q=5^2$}

Let $r=3$ and $s=14+q$. The base-$q$ digits of $s$ and $S_2$ are given below.
\[
\begin{tabular}{c|ccccc}
digit position &0 &1&\kern4em& & $e-2a$ \\ \hline
$s$ & 14& 1 \\
$S_2$ &3&  &&12
\end{tabular}
\]
We have
\[
C(N)=[(X_0+\cdots+X_{a-1})^{14}(X_1+\cdots+X_{a}):X_0^3X_{e-2a-1}^{12}]=\binom{14}3\not\equiv_5 0.
\]

\subsubsection{The case $q=5$}

Let $r=1$ and $s=3+q+q^{e-2a}$. The base-$q$ digits of $s$, $q^{e-2a}2r$ and $S_1$ are given below.
\[
\begin{tabular}{c|ccccccc}
digit position &0 &1&\kern4em& &$e-2a$ &\kern4em& $e-a-1$\\ \hline
$s$ & 3 &1 &&&1\\
$q^{e-2a}2r$ &&&&& 2\\
$S_1$ &3& & &1&1&&2
\end{tabular}
\]
We have
\begin{align*}
C(N)
&=[(X_0+\cdots+X_{a-1})^3(X_1+\cdots+X_{a})(X_{e-2a}+\cdots+X_{e-a-1})^3\cr
&\kern1em :X_0^3X_{e-2a-1}X_{e-2a}X_{e-a-1}^2]\cr
&=\binom 32\not\equiv_5 0.
\end{align*}

\subsubsection{The case $q=3$}

Let $r=2$ and $s=1+q^{e-2a}$. The base-$q$ digits of $s$ , $q^{e-2a}2r$ and $S_1$ are given below.
\[
\begin{tabular}{c|ccccccccc}
digit position &0 &&\kern4em& &$e-2a$& &\kern4em& $e-a-1$\\ \hline
$s$ & 1 & &&&1\\
$q^{e-2a}2r$ &&&&& 1&1\\
$S_1$ && & &1&0&1&&1&1
\end{tabular}
\]
We have
\begin{align*}
C(N)&=[(X_0+\cdots+X_{a-1})(X_{e-2a}+\cdots+X_{e-a-1})^2(X_{e-2a+1}+\cdots+X_{e-a})\cr
&\kern1em :X_{e-2a-1}X_{e-2a+1}X_{e-a-1}X_{e-a}]\cr
&=2\not\equiv_3 0.
\end{align*}

%%%%%%%%%%%%%%%%%%%%%%%%%%%%%%%%%%%%%%%%%%%
\subsection{The case $e=2a+1$}

\subsubsection{The case $q>4$}
 
Let $r=q-2$ and $s=(q-1)+3q$. The base-$q$ digits of $s$, $q^{e-2a}2r$ and $S_1$ are given below.
\[
\begin{tabular}{c|cccccc}
digit position &0 &1&&\kern4em& $e-a-1$\\ \hline
$s$ & $q-1$ &3 \\
$q^{e-2a}2r$ &&$q-4$& 1\\
$S_1$ &$q-1$&3 & &&$q-4$&1
\end{tabular}
\]

\noindent
We have
\begin{align*}
C(N)\,&=[(X_0+\cdots+X_{a-1})^{q-1}(X_1+\cdots+X_{a})^{q-1}(X_2+\cdots+X_{a+1})\cr
&\kern1em :X_0^{q-1}X_1^3X_{a}^{q-4}X_{a+1}]\cr
&=\binom{q-1}3\equiv_p -1\not\equiv_p 0.
\end{align*}

\subsubsection{The case $q=3$}

Let $r=2$ and $s=1+q$. First assume that $a\ge 3$. The base-$q$ digits of $s$, $q^{e-2a}2r$ and $S_1$ are given below.
\[
\begin{tabular}{c|cccccc}
digit position &0 &1&2&\kern4em& $e-a-1$\\ \hline
$s$ & 1&1 \\
$q^{e-2a}2r$ &&1&1\\
$S_1$ &1&0 & 1&&1&1
\end{tabular}
\]
Let
\[
F=(X_0+\cdots+X_{a-1})(X_1+\cdots+X_{a})^2(X_2+\cdots+X_{a+1}).
\]
Then $C(N)=[F:X_0X_2X_aX_{a+1}]=2\not\equiv_3 0$.

If $a=2$, the digits of $S_1$ are modified as follows.
\[
\begin{tabular}{c|cccc}
digit position &0 &$1$&$2$\\ \hline
$s$ & 1&1 \\
$q^{e-2a}2r$ &&1&1\\
$S_1$ &1&0&2 &1
\end{tabular}
\]
In this case, $C(N)=[F:X_0X_2^2X_3]=1\not\equiv_3 0$.

%%%%%%%%%%%%%%%%%%%%%%%%%%%%%%%%%%%%%%%%%%%
%  section 6
%%%%%%%%%%%%%%%%%%%%%%%%%%%%%%%%%%%%%%%%%%%
\section{The case $e/2<a< e$}

The method of the previous two sections works for the case $a+2\le e<2a$, but a slightly different method works better for the case $e=a+1$.

\subsection{The case $a+2\le e<2a$}
 
We will use the following notations. Let $0<r<q$ and $s>0$ be integers such that $s\equiv -r\pmod{q-1}$ and $s\equiv r/2\pmod q$, i.e., $s\equiv r((q^2+1)/2-2q)\pmod{q(q-1)}$. Let 
\begin{equation}\label{6.2}
N=rq^{e-a}+s.
\end{equation}
Write
\begin{equation}\label{6.15}
s=\sum_{j\ge 0}s_jq^j,
\end{equation}
where $s_j\in\{0,\dots,q-1\}$. Let
\begin{equation}\label{6.15.1}
S_2=2rq^{e-a-1}+q^{-1}(2s-r)
\end{equation}
and
\begin{equation}\label{6.17}
\mathcal E_2=\Bigl\{(\epsilon_0,\epsilon_1,\dots):\epsilon_j\ge 0,\ \sum_{j\ge 0}\epsilon_j=|s|_q,\ \sum_{j\ge 0}\epsilon_jq^j=S_2 \Bigr\}.
\end{equation}

\begin{lem}\label{L6.1}
If $s\le q^{e-a-1}(q-1)$, then $0<N<q^e-1$ and
\begin{equation}\label{6.4}
C(N)=\sum_{\alpha_{2i}}\binom{s}{\alpha_{21},\dots,\alpha_{2a}},
\end{equation}
where the sum is over all nonnegative integers $\alpha_{2i}$, $1\le i\le a$, subject to the conditions
\begin{gather}\label{6.5}
\sum_{i=1}^{a}\alpha_{2i}=s,\\
\label{6.6.1}
\sum_{i=1}^{a}\alpha_{2i}q^{i-1}=S_2.
\end{gather}
\end{lem}

\begin{proof}
Identical to the proof of Lemma~\ref{L5.1}.
\end{proof}

\begin{prop}\label{P6.2}
If $s\le q^{e-a-1}(q-1)$, then
\begin{equation}\label{6.20}
C(N)=\sum_{(\epsilon_0,\epsilon_1,\dots)\in\mathcal E_2}\Bigl[\prod_{j\ge 0}(X_j+X_{j+1}+\cdots+X_{j+a-1})^{s_j}:X_0^{\epsilon_0}X_1^{\epsilon_1}\cdots\Bigr].
\end{equation}
\end{prop}

\begin{proof}
Identical to the proof of Proposition~\ref{P4.2}.
\end{proof}

\begin{prop}\label{P6.3}
Assume that $a+2\le e<2a$. Let $u$ be a positive integer such that $6u-3<q$. Let
$r=q-2u$ and $s=q-u+(3u-2)q$.
Then
\begin{equation}\label{6.31}
C(N)=(-4u+1)\binom{-u}{6u-3}.
\end{equation}
\end{prop}

\begin{proof}
First note that $0<r<q$ and $0<s<q(q-1)\le q^{e-a-1}(q-1)$. Clearly,
\begin{equation}\label{6.32}
s_j=\begin{cases}
q-u&\text{if}\ j=0,\cr
3u-2&\text{if}\ j=1,\cr
0&\text{if}\ j>1.
\end{cases}
\end{equation}
We have 
\begin{equation}\label{6.33}
S_2=2rq^{e-a-1}+q^{-1}(2s-r)=6u-3+(q-4u)q^{e-a-1}+q^{e-a}=\sum_{j\ge 0}e_jq^j,
\end{equation}
where
\begin{equation}\label{6.34}
e_j=\begin{cases}
6u-3&\text{if}\ j=0,\cr
q-4u&\text{if}\ j=e-a-1,\cr
1&\text{if}\ j=e-a,\cr
0&\text{otherwise}.
\end{cases}
\end{equation}
Since $\sum_{j\ge 0}e_j=q+2u-2=|s|_q$, we have $\mathcal E_2=\{(e_0,e_1,\dots)\}$. By \eqref{6.20},
\begin{align*}
C(N)
&=[(X_0+\cdots+X_{a-1})^{q-u}(X_1+\cdots+X_{a})^{3u-2}:X_0^{6u-3}X_{e-a-1}^{q-4u}X_{e-a}]\cr
&=\binom{q-u}{6u-3}\binom{q-4u+1}1\equiv_p(-4u+1)\binom{-u}{6u-3}.
\end{align*}
\end{proof}

\begin{cor}\label{C6.4}
Assume that $a+2\le e<2a$ and that $f_{a,q}$ is a PP of $\f_{q^e}$. Then $q\in\{3,3^2\}$.
\end{cor}

\begin{proof}
By Hermite's criterion, $C(N)=0$. By \eqref{6.31},
\[
C(N)=\begin{cases}
3&\text{if $u=1$ and $q>3$},\cr
2\cdot 5\cdot 7&\text{if $u=2$ and $q>9$}.
\end{cases}
\]
The conditions $C(N)=0$ with $1\le u\le 2$ and $q>6u-3$ give the following possibilities for $p$ and $q$.
\[
\begin{tabular}{c|c|c|c}
range of $q$ & $u$ & $p$ & $q$ \\ \hline
$q\le 3$ & & & $3$\\
$3<q\le 9$ & $u=1$ & $3$ & $3^2$\\
$9<q$ & $1\le u\le 2$ & none & none \\
\end{tabular}
\]
Hence the claim.
\end{proof}

\subsubsection{The case $q=3^2$}

Let $r=1$ and $s=5+2q$. The base-$q$ digits of $s$ and $S_2$ are given below.
\[
\begin{tabular}{c|ccccc}
digit position &0 &1&\kern4em& &$e-a$\\ \hline
$s$ & 5&2  \\
$S_2$ &5&&&2
\end{tabular}
\]
We have 
\[
C(N)=[(X_0+\cdots+X_{a-1})^5(X_1+\cdots+X_{a})^2:X_0^5X_{e-a-1}^2]=1\not\equiv_30.
\]

\subsubsection{The case $q=3$}

Let $r=1$ and $s=2+2q+\cdots+2q^{e-a-2}+q^{e-a-1}$. The base-$q$ digits of $s$ and $S_2$ are given below.
\[
\begin{tabular}{c|ccccccc}
digit position &0 & $\cdots$ &&&$e-a-1$&\\ \hline
$s$ & 2&$\cdots$&2&2&1  \\
$S_2$ &2&$\cdots$&2&0&0&1
\end{tabular}
\]

\noindent
We have $\mathcal E_2=\{(\epsilon_{i0},\epsilon_{i1},\dots):1\le i\le \max\{1,e-a-2\}\}$, where
\[
\begin{array}{rccccccl}
&\scriptstyle 0&&&& &\scriptstyle e-a \cr   
(\epsilon_{10},\epsilon_{11},\dots)=(\kern-0.5em&2&\cdots&2&0&3&0&\kern-0.5em),
\end{array}
\]
and for $2\le i\le e-a-2$,
\[
\begin{array}{rcccccccccccl}
&\scriptstyle 0&&&\scriptstyle i-2& &&&&&&\scriptstyle e-a \cr   
(\epsilon_{i0},\epsilon_{i1},\dots)=(\kern-0.5em&2&\cdots&2&5&1&2&\cdots&2&0&0&1&\kern-0.5em),
\end{array}.
\]
Let 
\[
F=\Bigl(\prod_{i=0}^{e-a-2}(X_i+\cdots+X_{i+a-1})^2\Bigr)(X_{e-a-1}+\cdots+X_{e-2}).
\]
Then it is easy to see that
\[
[F:X_0^{\epsilon_{i0}}X_1^{\epsilon_{i1}}\cdots]=
\begin{cases}
1&\text{if}\ i=1,\cr
0&\text{if}\ 2\le i\le e-a-2.
\end{cases}
\]
Therefore $C(N)=\sum_{i=1}^{\max\{1,e-a-2\}}[F:X_0^{\epsilon_{i0}}X_1^{\epsilon_{i1}}\cdots]=1\not\equiv_3 0$.

%%%%%%%%%%%%%%%%%%%%%%%%%%%%%%%%%%%%%%%%%%%
\subsection{The case $e=a+1$}

\begin{lem}\label{L6.5}
Let $e\ge 3$ the nonnegative integer solutions $(c_1,\dots,c_{e-1})$ of the system
\begin{align}\label{e6.18}
&\sum_{j=1}^{e-1}c_j(q^j-2)\equiv 0\pmod{q^e-1},\\
\label{e6.19}
&\sum_{j=1}^{e-1}c_j=q^2-1
\end{align}
are as follows:
\begin{itemize}
\item[(i)] If $e\ge 5$,
\begin{equation}\label{e6.20}
\begin{array}{crcccccccl}   
(c_1,\dots,c_{e-1})&=(\kern-0.5em&q-1&1&0&\cdots&0&q&q^2-2q-1&\kern-0.5em)\ \text{or}\vspace{0.4em}\cr
&(\kern-0.5em&2q-1&0&0&\cdots&0&0&q^2-2q&\kern-0.5em).
\end{array}
\end{equation}

\item[(ii)] If $e=4$,
\begin{equation}\label{e6.21}
\begin{array}{crcccl}   
(c_1,c_2,c_3)&=(\kern-0.5em&q-1&q+1&q^2-2q-1&\kern-0.5em)\ \text{or}\vspace{0.4em}\cr
&(\kern-0.5em&2q-1&0&q^2-2q&\kern-0.5em).
\end{array}
\end{equation}  

\item[(iii)] If $e=3$,
\begin{equation}\label{e6.22}
(c_1,c_2)=(2q-1,q^2-2q).
\end{equation}
\end{itemize}
\end{lem}

\begin{proof}
Let $(c_1,\dots,c_{e-1})$ be a solution of \eqref{e6.18} and \eqref{e6.19}. Let
\begin{equation}\label{e6.23}
\sum_{j=1}^{e-1}c_j(q^j-2)=k(q^e-1).
\end{equation}
Then $k\equiv 2\sum_{j=1}^{e-1}c_j=2(q^2-1)\equiv -2\pmod q$. Since
\[
k(q^e-1)=\sum_{j=1}^{e-1}c_j(q^j-2)\le(q^2-1)(q^{e-1}-2)=q^{e+1}-2q^2-q^{e-1}+2<q^{e+1}-q,
\]
we have $k<q$. Hence $k=q-2$. Now \eqref{e6.23} becomes
\begin{equation}\label{e6.24}
\sum_{j=1}^{e-1}c_jq^j=q^{e+1}-2q^e+2q^2-q.
\end{equation}
Thus
\[
(q^2-1-c_{e-1})q^{e-2}+c_{e-1}q^{e-1}\ge q^{e+1}-2q^e+2q^2-q,
\]
which gives
\[
q^2-1-c_{e-1}+c_{e-1}q\ge q^3-2q^2+1.
\]
Hence
\[
(q-1)c_{e-1}\ge q^3-3q^2+2=(q-1)(q^2-2q-2),
\]
which gives $c_{e-1}\ge q^2-2q-2$. Rewrite \eqref{e6.24} as
\begin{equation}\label{6.25}
(c_1+1)q+c_2q^2+\cdots+c_{e-2}q^{e-2}+(c_{e-1}-(q^2-2q-2))q^{e-1}=2q^2+2q^{e-1}.
\end{equation}
Note that the sum of the coefficients of the left side of \eqref{6.25} is $2q+2$ and $c_1+1\ge 1$; hence these coefficients are obtained by borrowing twice from the base-$q$ digits of $2q^2+2q^{e-1}$, at least once from $2q^2$. Therefore
\begin{align*}
&(c_1+1,c_2,\dots,c_{e-2},c_{e-1}-(q^2-2q-2))\cr
&=
\begin{cases}
(q,1,0,\dots,0,q,1)\ \text{or}\ (2q,0,0,\dots,0,0,2)&\text{if}\ e\ge 5,\cr
(q,q+1,1)\ \text{or}\ (2q,0,2)&\text{if}\ e=4,\cr
(2q,2)&\text{if}\ e=3.
\end{cases}
\end{align*}
Hence the conclusion.
\end{proof}

Now let $N=q^2-1$, and note that $0<N<q^e-1$. We have
\[
C(N)=\sum_{c_1,\dots,c_{e-1}}\binom{q^2-1}{c_1,\dots,c_{e-1}},
\]
where $(c_1,\dots,c_{e-1})$ satisfies \eqref{e6.18} and \eqref{e6.19}. By Lemma~\ref{L6.5}, when $e\ge 5$,
\[
C(N)=\binom{q^2-1}{q-1,1,q,q^2-2q-1}+\binom{q^2-1}{2q-1,q^2-2q}\equiv_p -1\not\equiv_p 0;
\]
when $e=4$,
\[
C(N)=\binom{q^2-1}{q-1,q+1,q^2-2q-1}+\binom{q^2-1}{2q-1,q^2-2q}\equiv_p -1\not\equiv_p 0;
\]
when $e=3$,
\[
C(N)=\binom{q^2-1}{2q-1,q^2-2q}\equiv_p -1\not\equiv_p 0.
\]

%%%%%%%%%%%%%%%%%%%%%%%%%%%%%%%%%%%%%%%%%%%
%  section 7
%%%%%%%%%%%%%%%%%%%%%%%%%%%%%%%%%%%%%%%%%%%
\section{The case $e<a<(p-1)e$}

Write $a=ue+v$, where $1\le u\le p-2$ and $1\le v\le e$. Then
\begin{align}\label{7.1}
f_{a,q}(X)\,&\equiv uX^{q^e-2}+\cdots+uX^{q^{v+1}-2}+(u+1)X^{q^{v}-2}+\cdots+(u+1)X^{q-2}\\
&\pmod{X^{q^e}-X}.\nonumber
\end{align}

\begin{lem}\label{L7.1}
Let $e\ge 2$. The only nonnegative integer solution of the system
\begin{align}\label{7.2}
&\sum_{j=1}^ec_j(q^j-2)\equiv 0\pmod{q^e-1},\\
\label{7.3}
&\sum_{j=1}^ec_j=q-1
\end{align} 
is $(c_1,\dots,c_e)=(1,0,\dots,0,q-2)$. 
\end{lem}

\begin{proof}
Let $(c_1,\dots,c_e)$ satisfy \eqref{7.2} and \eqref{7.3}. Let
\begin{equation}\label{7.5}
\sum_{j=1}^ec_j(q^j-2)=k(q^e-1).
\end{equation}
Then $k\equiv-2\pmod q$. Since
\[
k(q^e-1)=\sum_{j=1}^ec_j(q^j-2)\le (q-1)(q^e-2)<q(q^e-1),
\]
we have $k<q$, and hence $k=q-2$. Now \eqref{7.5} becomes
\begin{equation}\label{7.6}
c_1q+\cdots+c_{e-1}q^{e-1}+(c_e+2)q^e=q+q^{e+1}.
\end{equation}
Note that the sum of the coefficients of the left side of \eqref{7.6} is $q+1$. Thus we must have $(c_1,\dots,c_{e-1},c_e+2)=(1,0,\dots,0,q)$. Hence the conclusion.
\end{proof}

By \eqref{7.1} and Lemma~\ref{L7.1}, we have 
\[
C(q-1)=\binom{q-1}1(u+1)u^{q-2}\equiv_p-(u+1)u^{q-2}\not\equiv_p 0.
\]

%%%%%%%%%%%%%%%%%%%%%%%%%%%%%%%%%%%%%%%%%%%
%  section 8
%%%%%%%%%%%%%%%%%%%%%%%%%%%%%%%%%%%%%%%%%%%
\section{The case $(p-1)e<a\le pe-2$}

Write $a+1=(p-1)e+k$, where $1\le k<e$. Then
\begin{equation}\label{8.1}
f_{a,q}(X)\equiv -X^{q^e-2}-\cdots-X^{q^k-2}\pmod{X^{q^e}-X}.
\end{equation}

\begin{lem}\label{L8.1}
The only nonnegative integer solution of the system
\begin{align}\label{8.2}
&\sum_{j=k}^ec_j(q^j-2)\equiv 0\pmod{q^e-1},\\
\label{8.3}
&\sum_{j=k}^ec_j=q^k-1
\end{align}
is $(c_k,\dots,c_e)=(1,0,\dots,0,q^k-2)$.
\end{lem}

\begin{proof}
Let $(c_k,\dots,c_e)$ be a solution of \eqref{8.2} and \eqref{8.3}. Then
\[
0\equiv\sum_{j=k}^ec_jq^j-2(q^k-1)\equiv \sum_{j=k}^ec_jq^j-2(q^k-q^e)\pmod{q^e-1},
\]
and hence
\begin{equation}\label{8.4}
\sum_{j=k}^ec_jq^{j-k}+2q^{e-k}-2\equiv 0\pmod{q^e-1}.
\end{equation}
Clearly, $c_e\le q^k-2$. (Otherwise, $c_k=\cdots=c_{e-1}=0$ and $c_e=q^k-1$. Then \eqref{8.2} is not satisfied.) If $c_e\le q^k-3$, then
\begin{align*}
0\,&<\sum_{j=k}^ec_jq^{j-k}+2q^{e-k}-2\le (q^k-3)q^{e-k}+2q^{e-k-1}+2q^{e-k}-2\cr
&=q^e-q^{e-k}+2q^{e-k-1}-2<q^e-1,
\end{align*}
which is a contradiction to \eqref{8.4}. Hence we have proved that $c_e=q^k-2$. Now \eqref{8.4} becomes
\[
\sum_{j=k}^{e-1}c_jq^{j-k}\equiv 1\pmod{q^e-1}.
\]
Since $\sum_{j=1}^{e-1}c_j=1$, we must have $c_k=1$ and $c_{k+1}=\cdots=c_{e-1}=0$.
\end{proof}

By \eqref{8.1} and Lemma~\ref{L8.1}, we have
\[
C(p^k-1)=(-1)^{q^k-1}\binom{q^k-1}1\equiv_p -1\not\equiv_p 0.
\]

%%%%%%%%%%%%%%%%%%%%%%%%%%%%%%%%%%%%%%%%%%%

\end{document}